\documentclass[a4paper,11pt]{article}

\usepackage[english]{babel}
\usepackage[a4paper]{geometry}
\usepackage{enumerate}
\usepackage{amsmath}
\usepackage{amsthm}
\usepackage{amssymb}
\usepackage{hyperref}
\usepackage{amssymb}
\usepackage{color}
\usepackage[utf8]{inputenc}
\usepackage{graphicx}
\usepackage{tikz}
\usepackage{tikz-cd}
\usepackage{verbatim} 
\usepackage{mathtools} 
\usepackage{mathabx} 
\usepackage{hyperref}
\usepackage{comment}
\usetikzlibrary{decorations.markings}
\usepackage{bbm}
\usepackage{xfrac} 
\usepackage{microtype} 

	\newtheorem{theorem}{Theorem}[section]  
	\newtheorem{cor}[theorem]{Corollary}
	
	\newtheorem{lemma}[theorem]{Lemma}

	\numberwithin{equation}{section}
    \theoremstyle{definition}
    \newtheorem{defn}{Definition}
    \newtheorem{rem}[theorem]{Remark}
	
	\newcommand{\mO}{\mathcal{O}}  
    \newcommand{\OS}{\mathcal{O}_S}  
    
    \newcommand{\Div}{\text{Div}}
    \newcommand{\supp}{\text{supp}}
    \newcommand{\te}{W} 
    \newcommand{\Mod}[1]{\ (\mathrm{mod}\ #1)}
    \newcommand{\disc}{\text{Disc}}

    \newcommand{\Mat}{\text{Mat}}
    \newcommand{\LD}{\mathcal{L}(D)}
    \newcommand{\inv}{\text{inv}}
    
\begin{document}
\title{Local to global principle for higher moments over function fields}

\author{Andy Hsiao$^1$, Junhong Ma Blackmer$^1$, Severin Schraven$^1$, Ying Qi Wen$^{2}$ \\
\\
$^1$ Department of Mathematics, University of British Columbia, \\
Vancouver, BC V6T 1Z2, Canada\\\\
$^2$ Department of Electrical and Computer Engineering, University of British Columbia, \\
Vancouver, BC V6T 1Z4, Canada}

\date{}
 
	
	
	\maketitle
	\begin{abstract}
	We establish a local to global principle for higher moments over holomorphy rings of global function fields and use it to compute the higher moments of rectangular unimodular matrices and Eisenstein polynomials with coefficients in such rings. 
	\end{abstract}

\section{Introduction}

A classical problem in number theory is to compute the natural density of subsets of the integers. The natural density of a subset $A \subseteq \mathbb{Z}^d$ is given by considering the number of points in $A\cap [-H;H)^d$, normalizing it by $(2H)^d$ - the number of points in the whole ``box" $[-H;H)^d$ - and then taking $H\rightarrow \infty$. A convenient tool to compute such densities in special situations was developed in \cite[Lemma 20]{poonenAnn}. If the set $A$ can be characterised locally in the sense that $A = \bigcap_p (\mathbb{Z}^d \cap U_p)$ for some defining sets $U_p \subseteq Z_p^d$ (which is the case if $A$ is defined by equations modulo prime powers), then under certain
conditions the natural density of $A$ can be expressed in terms of the Haar measures of the $U_p$. One would naturally expect that such a local to global principle for natural density should hold for any global field. Indeed, a similar result was established for number fields in \cite[Proposition 3.2]{bright2016failures}. Finally, the case of global function fields was covered in \cite[Theorem 2.1]{functionfields}. 

One should think of the natural density as a substitute for a finite Haar measure on the integers. It then becomes natural to ask whether one can make sense of the notion of expected value (or any higher moment). If $U_p\subseteq \mathbb{Z}_p^d$ are again the defining sets of our set in $\mathbb{Z}^d$, then expected value is defined as
\begin{align*}
    \lim_{H\rightarrow \infty} \sum_{a\in [-H;H)^d} \frac{\vert \{ p \ : \ p \text{ prime}, \ a\in U_p \} \vert}{(2H)^d}.
\end{align*}
This means, for the elements in the box $[-H;H)^d$ we count in how many defining sets it is contained, average over the total number of points in the box and let the side length of the box go to infinity. This notion was considered in \cite{MMRW} for Eisenstein polynomials. 

One can make a similar definition of expected value over the ring of algebraic integers of number fields or for higher moments. In \cite{MSW}, \cite{MSTW} a local to global principle for higher moments over number fields has been established based on a local to global principle for natural densities \cite{poonenAnn, bright2016failures}. In \cite{bib:HolMS} the notion of natural density for holomorphy rings of global function fields has been introduced (slightly different from the one in \cite{bib:poonen}), where the boxes are replaced by Riemann-Roch spaces (see Section $2$ for precise definitions). For further generalization and overview we directed the interested reader to \cite{demangos2020densities}. In this paper we will prove a local to global principle for higher moments over global function fields and show how to use this tool for some interesting examples.

This paper is organized as follows: In Section $2$ we recall the local to global principle for the natural density in global function fields as introduced in \cite{functionfields}. In Section $3$ we will prove our main theorem, the local to global principle for higher moments over function fields and in Section $4$ we will apply it to some examples (coprime pairs, affine Eisenstein polynomials and rectangular unimodular matrices, all with coefficients in the holomorphy ring of some global function field).

\section{Preliminaries}

In this section we recall the basic definitions and results from \cite{functionfields}. We follow the terminology of \cite{bib:stichtenoth2009algebraic} for function fields and related concepts.

Let $F$ be a global function field, that is a finite extension $F/\mathbb{F}_q(X)$, where $\mathbb{F}_q$ denotes a finite field with $q$ elements.
We denote by $\mO_P$ a valuation ring of $F$, having maximal ideal $P$. The set of all such places of $F$ will be denoted by $\mathbb{P}_F$.
If $\emptyset \neq S \subsetneq \mathbb{P}_F$ and $t\in \mathbb{N}$, we define 
\begin{equation} \label{def:St}
    S_t = \{ P\in S \ : \ \text{deg}(P) \geq t \}.
\end{equation}
Moreover, we write $\mO_S$ to denote the holomorphy ring of $S$
\[\mO_S=\bigcap_{P\in S} \mO_P.\] 
The easiest example of holomorphy ring is $\mathbb{F}_q[x]$ as this consists of the intersection of all the valuation rings of $\mathbb{F}_q(x)$ different from the infinite place (see \cite[Section 1.2]{bib:stichtenoth2009algebraic}).
We denote by $\Div(F)$ the set of divisors of $F$, i.e. the free abelian group on the set $\mathbb{P}_F$. Furthermore, 
for $D=\sum_{P\in \mathbb{P}_F} n_P P\in \Div(F)$ we define $v_P(D)=n_P$. The support  $\supp(D)$ of $D$ is defined as the finite subset of $\mathbb{P}_F$ for which $v_P(D)$ is non-zero. For $D, \widetilde{D} \in \Div(F)$ we write $D\leq \widetilde{D}$ if and only if $v_P(D) \leq v_P(\widetilde{D})$ for all $P\in \mathbb{P}_F$. Note that this defines a partial order on $\Div(F)$.
Moreover, we will write $D\geq 0$ whenever $v_P(D)\geq 0$ for all $P$ in $\mathbb{P}_F$. Let 
\[\Div^+(F)=\{D\in \Div(F)\;|\; D\geq 0\}.\]
For $S\subseteq \mathbb{P}_F$ let $\mathcal{D}_S$ be the subset of divisors of $\Div^+(F)$ having support contained in the complement of $S$.

Let $(a_D)_{D\in \mathcal{D}_S} \subseteq \mathbb{R}$, then we write for $a\in \mathbb{R}$
\begin{align*}
    \lim_{D\in \mathcal{D}_S} a_D =a
\end{align*}
if for every $\varepsilon >0$ there exists $D_\varepsilon \in \mathcal{D}_S$ such that for all $D \in \mathcal{D}_S$ with $D\geq D_\varepsilon$ one has $\vert a-a_D \vert < \varepsilon$. Similarly one defines $\limsup_{D\in \mathcal{D}_S} a_D$ and $\liminf_{D\in \mathcal{D}_S} a_D$.
For further information on Moore--Smith convergence see \cite[Chapter 2]{bib:kelley1955general}. 

We define the upper density for $A \subseteq \mO_S^d$ as follows:
\[\overline \rho_S(A):=\limsup_{D\in\mathcal{D}_{S}}\frac{|A\cap \mathcal{L}(D)^d|}{q^{\ell(D)d}},\]
where $\mathcal{L}(D)$ is the Riemann-Roch space attached to the divisor $D$ and $\ell(D)=\dim_{\mathbb{F}_q}(\mathcal{L}(D))$.
Analogously, one can give a notion of lower density $\underline \rho_S$ by replacing the limes superior by the limit inferior. Whenever these two quantities coincide, we define the density of $A$ as $\overline \rho_S(A)=\underline \rho_S(A)=\rho_S(A)$. 

This definition of density coincides with the classical definition, 
\begin{align*}
    \lim_{d\rightarrow \infty} \frac{\vert A \cap \{ f\in \mathbb{F}_q[x] \ : \ \deg(f) \leq d \} \vert}{ \vert \{ f\in \mathbb{F}_q[x] \ : \ \deg(f) \leq d \} \vert},
\end{align*}
when $\mathcal{O}_S=\mathbb{F}_q[x]$ (that is $F=\mathbb{F}_q(x)$ and $S$ is all the places except the infinite place). This was used in \cite{bib:poonen} to compute the density of square-free, multivariate polynomials with coefficients in $\mathbb{F}_q[x]$.
The reader should notice that in the case of the coordinate ring of the $\mathbb{P}^d$ a similar notion of density was provided in \cite{poonen2004bertini}.

For a valuation ring $\mO_P$ let us denote by $\widehat{\mO}_P$ its completion. As $F$ is a global function field, $\widehat{\mO}_P$ admits a normalized Haar measure, which we denote by $\mu_P$. By abuse of notation we will denote the product measure on $\widehat{\mathcal{O}}_P$ also by $\mu_P$. 
For $U\subseteq \widehat{\mO}_P^n$ we denote by $\partial U$ the boundary of $U$ with respect to the $P$-adic metric.

In \cite{functionfields} the following local to global principle for densities over global function fields was shown (this is an extension of the work of \cite{bib:poonen}).

\begin{theorem}[\text{\cite[Theorem 2.1]{functionfields}}]\label{thm:densityFF}
Let $d$ be a positive integer, $S$ be a proper, nonempty subset of places of $F$, $S_t$ as defined in \eqref{def:St} and $\mO_S$ the holomorphy ring of $S$. 
For any $P\in S$, let $U_P\subseteq \widehat{\mO}_P^d$ be a Borel-measurable set such that $\mu_P(\partial U_P)=0$.
Suppose that
\begin{equation}\label{fund_cond_loc_to_glob}
\lim_{t\rightarrow \infty }\overline{\rho}_S(\{a\in \mO_S^d\:|\: a\in U_P \;\text{for some }\; P\in S_t\})=0.
\end{equation}
Let $\pi:\mO_S^d\longrightarrow 2^{S}$ be defined by
$\pi(a)=\{P\in S: a\in U_P\}\in 2^{S}$. Then
\begin{itemize}
\item[(i)] $\displaystyle{\sum_{P\in S} \mu_P(U_P)}$ is 
convergent.
\item[(ii)] Let $\Gamma\subseteq 2^S$. Then $\nu (\Gamma):=\rho_S(\pi^{-1}(\Gamma))$ exists and  $\nu$ defines a measure on $2^{S}$.
\item[(iii)] $\nu$ is concentrated at finite subsets of $S$. In addition, if $T\subseteq S$ is finite we have:
\begin{equation} \label{general product formula}
    \nu(\{T\})=\left(\prod_{P\in T}\mu_P(U_P)\right)\prod_{P\in S\setminus T } (1-\mu_P(U_P)).
\end{equation}
\end{itemize}
\end{theorem}

In the same paper, the following variant of Ekedahl's sieve was proved as a tool to verify assumption \eqref{fund_cond_loc_to_glob}.

\begin{theorem}[\text{\cite[Theorem 2.2]{functionfields}}]\label{condition_verified_polynomials_THEOREM}
Let $F$ be a global function field and $S$ be a proper, nonempty subset of $\mathbb{P}_F$. Let $\mO_S$ be the holomorphy ring of $S$. Let $f,g \in \mO_S[x_1,\dots,x_d]$ be coprime polynomials. Then
\begin{equation}\label{condition_verified_polynomials}
\lim_{t\rightarrow \infty} \overline{\rho}_S\left(\{y\in \mO_S^d: \quad f(y)\equiv g(y)\equiv 0 \mod P\quad  \text{for some}\; P\in S_t\}\right)=0.
\end{equation}
\end{theorem}

\begin{rem}
    Note that in \cite{functionfields} the theorem is only stated for sets $S$ having finite complement, but the assumption is not needed in the proof. An alternative proof was sketched in \cite[Theorem 8.1]{bib:poonen}.
\end{rem}

The next Corollary follows from Theorem \ref{thm:densityFF}. It relates the density of the defining sets to their Haar measures and will play a crucial role in the proof of our main theorem.
\begin{cor}\label{denofUW}
    Let $F$ be a global function field with full field of constant equal to $\mathbb{F}_q$. Let $d$ be a positive integer, $S$ be a proper, nonempty subset of the places of $F$ and $\mathcal{O}_S$ the holomorphy ring of $S$. Let $P_1, . . . , P_n \in S$ be
distinct and for $j = 1, . . . , n$, let $U_{P_j} \subseteq \widehat{\mathcal{O}}_{P_j}^d$ be a Borel-measurable set with $\mu_{P_j} (\partial U_{P_j} ) = 0$. Then
	\begin{equation} \label{product formula}
		\rho_S\left( \bigcap_{j=1}^n \left( U_{P_j} \cap \mO_S^d \right)\right) = \prod_{j=1}^n \mu_{P_j}(U_{P_j}).
	\end{equation}
\end{cor} 

\begin{proof}
    Define $V_P=U_P$ for $P\in \{P_1, \dots, P_n\}$ and $V_P=\emptyset$ otherwise. Then $(V_P)_{P\in S}$ satisfies the assumption of Theorem \ref{thm:densityFF} and \ref{product formula} corresponds to \ref{general product formula} with $T=\{P_1, \dots, P_n\}$.
\end{proof}

\section{Higher moments}

For a family $(U_P)_{P \in S}$, we would like to define the expected value of the number of sets $U_P$ such that a random element $a \in \mO_S^d$ lies in and also the corresponding higher moments. 

\begin{defn} \label{definition moment}
	Let $F$ be a global function field with full field of constant equal to $\mathbb{F}_q$. Let $n$ and $d$ be positive integers, $\emptyset \neq S \subsetneq \mathbb{P}_F$. Suppose $U_P \subseteq \widehat{\mathcal{O}}_P$.
    Then we define \emph{the $n$-th moment of the system $(U_P)_{P \in S}$} to be 
	\begin{equation} \label{def:mun}
		\mu_n  =  \lim\limits_{D \in \mathcal{D}_S} \displaystyle{\frac{\sum\limits_{a \in \mathcal{L}(D)^d} \lvert \{ P \in S \mid a \in U_{P} \}\rvert^n} {q^{\ell(D)d}}},
	\end{equation} if it exists. 
        We call $\mu_1$ the expected value of the system $(U_P)_{P \in S}$.
\end{defn}

Our main theorem gives a way to compute higher moments for certain systems. 

\begin{theorem}\label{Thm:highermoments}
		Let $d$ and $n$ be positive integers. Let $F$ be a global function field with full field of constant equal to $\mathbb{F}_q$, $S$ be a proper, nonempty subset of $\mathbb{P}_F$, and $\mO_S$ be the holomorphy ring of $S$. For each  $P \in S$, let $U_P \subseteq \widehat{\mO}_P^d$ be a measurable set such that $\mu_P(\partial(U_P)) =0$. Let $S_t \coloneqq \{ P \in S \mid \deg{(P)} \geq t\}$.
		If 
		\begin{equation}\label{densitycond}
		    \lim_{t \to \infty} \overline{\rho}_S \left( \{ a \in \mO_S^d \mid a\in U_P \text{ for some } P \in S_t \}\right)=0
		\end{equation}
		is satisfied, and for some $\alpha \in [0, \infty)$ there exist absolute constants $c', c, \Tilde{c} \in \mathbb{Z}$, such that for all $D \in \mathcal{D}_S$ with $\deg(D) \geq \tilde{c}$ and for all $a \in \mathcal{L}(D)^d$ one has that
		\begin{equation}\label{newcond}
		   \left\vert \left\{ P\in S  \mid \deg(P) > c'\deg(D)^\alpha,  a \in U_P \cap \mathcal{L}(D)^d  \right\} \right\vert <c
		\end{equation} and  that there exists a sequence  $(v_P)_{P\in S}$, such that for all $m\in \{ 1, \dots, n\}$ and all $\deg(P_1), \dots, \deg(P_m) \leq c'\deg(D)^\alpha$ with $P_j\in S$ pairwise distinct one has that
		\begin{eqnarray}
		\left\vert \bigcap_{j=1}^{m} U_{P_j} \cap\mathcal{L}(D)^d  \right\vert & \leq q^{\ell(D)d} \prod\limits_{j=1}^{m} v_{P_j}, 
		\label{newcond2} \\
		\sum_{P \in S} v_P & \text{converges}, \label{newcond3}
		\end{eqnarray}
		then it follows that
		\begin{align*}
		\mu_{n}  &=  \lim\limits_{D \in \mathcal{D}_S} \displaystyle{\frac{\sum\limits_{a \in \mathcal{L}(D)^d 
				  } \lvert \{ P \in S \mid a \in U_{P} \}\rvert^n} {q^{\ell(D)d}}}
		\end{align*}
		exists and $\mu_{n}<\infty$. 
            For $l\in \mathbb{N}_{\geq 1}$ we denote by $\begin{Bmatrix} n \\ l \end{Bmatrix}$ the number of partitions of the set $\{1, \dots, n\}$ with exactly $l$ subsets.
            Then we have the formula
		\begin{equation}\label{mean}
		      \mu_{n} = \sum_{l=1}^n \begin{Bmatrix}
		          n \\ l
		      \end{Bmatrix} \sum_{\substack{P_1, \dots, P_{l}\in S  \\ \forall i<j \in \{1, \dots, l\}, \ P_i \neq P_j}} \prod\limits_{m=1}^{l} \mu_{P_m}(U_{P_m}).
		\end{equation}
	\end{theorem}
        \begin{rem}
        \begin{enumerate}
            \item Let $(U_P)_{P\in S}, (\widetilde{U}_P)_{P\in S}$ be systems satisfying the assumptions of Theorem \ref{Thm:highermoments} for a moment $r$. If $\mu_P(U_P) =\mu_P(\widetilde{U}_P)$ and we have another system $(V_P)_{P\in S}$ such that $U_P \subseteq V_P \subseteq \widetilde{U}_P$, then the $r$th moment of $(V_P)_{P\in S}$ exists too and is given by \eqref{mean}.
            \item The coefficient $\begin{Bmatrix} n\\ l \end{Bmatrix}$ in \eqref{mean} is the Stirling number of the second kind and can be computed as follows
            \begin{align*}
                \begin{Bmatrix}
                    n \\ l
                \end{Bmatrix} = \frac{1}{l!} \sum_{k=0}^l (-1)^k \binom{l}{k} (l-k)^n.
            \end{align*}

            \item One could weaken the assumptions \eqref{newcond}, \eqref{newcond2} and \eqref{newcond3}. Namely, it would be enough to assume that for every $D\in \mathcal{D}_S$ there exists $\widetilde{D}\in \mathcal{D}_S$ with $\widetilde{D}\geq D$ such that for all $D'\in \mathcal{D}_S$ with $D'\geq \widetilde{D}$ the conditions \eqref{newcond}, \eqref{newcond2} and \eqref{newcond3} hold true. The statement could be proved using the same ideas as in the proof below. 
            \item We could also include the case $\alpha=\infty$, meaning that we could drop assumption \eqref{newcond} and require instead \eqref{newcond2} and \eqref{newcond3} to hold for all places in $S$.
        \end{enumerate}
            
        \end{rem}
	
	\begin{proof}
	    We fix $S$ throughout the proof, and write $\mO$, $\rho$, and $\mathcal{D}$ in place of $\mO_S$, $\rho_S$, and $\mathcal{D}_S$ respectively.
	
		For $a\in \mO^d$ and $P\in \mathbb{P}_F$, we define
		\begin{align*}
		\tau(a,P) = \begin{cases}
		1,& a\in U_P,\\
		0,& \text{else}.
		\end{cases}
		\end{align*}
		For $M\in \mathbb{N},$ we have that
		$$\sum\limits_{a\in \mathcal{L}(D)^d} \frac{\left( \sum\limits_{P \in S} \tau(a,P) \right)^n}{q^{\ell(D)d}}
		=\sum\limits_{j=0}^n \binom{n}{j}  R_j(M,D),$$
		where for all $j \in \{0, \ldots, n\}$, we define
		\begin{align*}
		R_j(M,D) \coloneqq \sum_{a\in \mathcal{L}(D)^d} \frac{\left( \sum\limits_{P\in S_M} \tau(a,P) \right)^{n-j} \left( \sum\limits_{P \in S, \ \deg(P) < M} \tau(a,P)\right)^j}{q^{\ell(D)d}}.
		\end{align*}
		First we show that for all $j\in \{0, \dots, n-1\}$ the terms $R_j(M,D)$ are negligible for $M$ going to infinity. We define
		\begin{align*}
		l_{a,D} \coloneqq \left\vert \left\{ P \in S  \mid \deg(P) > c'\deg(D)^\alpha, a \in U_P \cap \mathcal{L}(D)^d  \right\} \right\vert.
		\end{align*}
		Then by \eqref{newcond} there exists $c, \tilde{c}>0$ such that for all $a\in \mO^d$ and all $D \in \mathcal{D}$ with $\deg(D) \geq \tilde{c}$ holds $l_{a,D} \leq c$. Thus, we get for $M\geq \tilde{c}$ 
		\begin{align*}
	    & \Theta_n(M,D)\coloneqq  q^{-\ell(D)d}\sum\limits_{a\in \mathcal{L}(D)^d} \left( \sum\limits_{P\in S_M} \tau(a,P) \right)^n \\
		= & q^{-\ell(D)d}\sum\limits_{i=0}^n \binom{n}{i} \hspace{-0.75cm}  \sum_{ \substack{a\in \mathcal{L}(D)^d: \\ a \in \bigcup\limits_{P_1, \dots, P_n \in S_M} \bigcap\limits_{j=1}^n U_{P_j}}} \hspace{-0.5cm} \left\vert \left\{ (P_j)_{j=1}^n \in S^n:  
		\begin{matrix}
		M < \deg(P_k) \leq c'\deg(D)^\alpha & \text{ if } 1 \leq k \leq i \\
		c'\deg(D)^\alpha < \deg(P_k) & \text{ if } i < k \leq n
		\end{matrix}
		\right\} \right\vert  \\
		\leq & q^{-\ell(D)d}\sum_{i=0}^n \binom{n}{i} \hspace{-0.75cm} \sum\limits_{\substack{a\in \mathcal{L}(D)^d: \\ a\in  \bigcup\limits_{P_1, \dots, P_n \in S_M} \bigcap\limits_{j=1}^n U_{P_j}}} \hspace{-0.75cm} l_{a,D}^{n-i}  \left\vert \{ (P_j)_{j=1}^i \in S^i : M < \deg(P_j) < c'\deg(D)^\alpha  \} \right\vert\\
		\leq  & q^{-\ell(D)d}c^{n} \left\vert \mathcal{L}(D)^d \cap  \bigcup\limits_{P\in S_M} U_P \right\vert 
		+q^{-\ell(D)d}\sum\limits_{i=1}^n \binom{n}{i} c^{n-i}\hspace{-1.5cm}  \sum\limits_{\substack{(P_1, \dots, P_i)\in S^i \\ M < \deg(P_1), \dots, \deg(P_i) \leq c'\deg(D)^\alpha}} \hspace{-1cm}\vert \mathcal{L}(D)^d \cap \bigcap\limits_{j=1}^i U_{P_j} \vert.
		\end{align*}
		Using  \eqref{newcond2}, we further have that
		\begin{align*}
		\Theta_n(M,D) \leq & c^{n} \frac{\vert \mathcal{L}(D)^d \cap \bigcup\limits_{P \in S_M} U_P \vert}{q^{\ell(D)d}} \\ &
		+ \sum\limits_{i=1}^n \sum_{j=1}^i 2^i \binom{n}{i} c^{n-i} \sum\limits_{\substack{(P_1, \dots, P_j)\in S^j \\ M < \deg(P_1) , \dots , \deg(P_j) < c'\deg(D)^\alpha}} \left( \prod_{k=1}^{j} v_{P_k} \right)  \\
		\leq  & c^{n} \frac{\vert \mathcal{L}(D)^d \cap \bigcup\limits_{P\in S_M} U_P \vert}{q^{\ell(D)d}} \\ &
		+ \sum\limits_{i=1}^n \sum_{j=1}^i 2^i \binom{n}{i} c^{n-i} \left( \sum_{P \in S_M} v_P \right)^j.
		\end{align*}
		
		This implies that
		\begin{align*}
		\limsup_{D \in \mathcal{D}} \Theta_n(M,D)
		\leq  & c^{n} \overline{\rho}\left(  \mO^d \cap \bigcup\limits_{P\in S_M} U_P \right)  \\ &
		+ \sum\limits_{i=1}^n \sum_{j=1}^i 2^i \binom{n}{i} c^{n-i} \left( \sum_{P\in S_M} v_P \right)^j.
		\end{align*}
		Thus, we get from \eqref{densitycond} and \eqref{newcond3}
		\begin{equation} \label{Sn}
		\lim_{M \rightarrow \infty} \limsup_{D\in \mathcal{D}} \Theta_n(M,D) =0.
		\end{equation}

		Using Hölder's inequality, we get for $j\in \{1, \dots, n-1\}$
		\begin{align*}
		R_j(M,D) \leq \Theta_n(M,D)^{(n-j)/n}  R_n(M,D)^{j/n}.
		\end{align*}
            Hence, if we can show that $\lim\limits_{M \rightarrow \infty} \lim\limits_{D\in\mathcal{D}} R_n(M,D)$ exists, then
            we get for all $j\in \{0, \dots, n-1\}$ 
		\begin{equation*}
		\lim_{M \rightarrow \infty}
		\limsup_{D \in \mathcal{D}}
		R_j(M,D) = 0
		\end{equation*}
		 and thus, $\mu_{n}$ exists as well and we have $$\mu_{n} = \lim\limits_{M \rightarrow \infty} \lim\limits_{D\in\mathcal{D}} R_n(M,D).$$ 
	In order to show that $\lim\limits_{M \rightarrow \infty} \lim\limits_{D\in\mathcal{D}} R_n(M,D)$ exists, we need to evaluate expressions of the form 
        \begin{align*}
            \lim_{D \in \mathcal{D}} \frac{\vert \mathcal{L}(D)^d \cap \bigcap_{j=1}^n U_{P_j}\vert }{ q^{\ell(D)d}}.
        \end{align*}
        We would like to use Corollary \ref{denofUW}, however, this only applies if the places are pairwise distinct. 
	For $l\in \mathbb{N}_{\geq 1}$, we denote by $\begin{Bmatrix} n\\ l\end{Bmatrix}$ the number of partitions of $\{1, \dots, n\}$ which contain exactly $l$ subsets. Note that $\{P_1, \dots, P_n\} = \{Q_1, \dots, Q_l\}$ with $Q_1, \dots , Q_k$ pairwise distinct, if and only if
        \begin{align*}
            \{ 1, \dots, n \} = \bigsqcup_{k=1}^l \{ j\in \{ 1, \dots, n\} \ : \ P_j =Q_k \}.
        \end{align*}    
        Thus, we get
        \begin{align*}
            R_n(M,D) &=  \sum\limits_{\substack{P_1, \dots, P_{\ell(\tau)} \in S \\ \deg(P_1), \dots, \deg(P_{l(\tau)})< M}} \frac{\vert \mathcal{L}(D)^d \cap \bigcap\limits_{j=1}^n U_{P_j} \vert}{q^{\ell(D)d}} \\
            &= \sum_{l=1}^n \begin{Bmatrix} n\\l \end{Bmatrix} \sum\limits_{\substack{P_1, \dots, P_{l} \in S \\ \deg(P_1), \dots, \deg(P_{l})< M  \\ \forall i<j\in \{1, \dots, l(\tau)\}, \ P_i \neq P_j}} \frac{\vert \mathcal{L}(D)^d \cap \bigcap\limits_{j=1}^{l} U_{P_j} \vert}{q^{\ell(D)d}}.
        \end{align*}
        As all the sums are finite, we can pull the limit over the divisors inside of the sums and get with
        Corollary \ref{denofUW}
        \begin{align*}
            \lim_{D \in \mathcal{D}} R_n(M,D) &= \sum_{l=1}^n \begin{Bmatrix} n \\ l \end{Bmatrix} \sum\limits_{\substack{P_1, \dots, P_{\ell(\tau)} \in S \\ \deg(P_1), \dots, \deg(P_{l(\tau)})< M  \\ \forall i<j\in \{1, \dots, l(\tau)\}, \ P_i \neq P_j}} \rho \left( \bigcap\limits_{j=1}^{l(\tau)} U_{P_j} \cap \mO^d \right) \\
            &= \sum_{l=1}^n \begin{Bmatrix} n \\ l \end{Bmatrix} \sum\limits_{\substack{P_1, \dots, P_{l} \in S \\ \deg(P_1), \dots, \deg(P_{l}) < M  \\ \forall i<j\in \{1, \dots, l(\tau)\}, \ P_i \neq P_j}} \prod\limits_{j=1}^{l} \mu_{P_j}(U_{P_j}).
        \end{align*}
	   Taking $M\rightarrow \infty$ yields \eqref{mean}. 
	Using \eqref{newcond3} and the crude estimate $\begin{Bmatrix} n \\ l \end{Bmatrix} \leq n^n$, one gets
		\begin{align*}
		    \mu_{n} 
		    \leq n^{n+1} \left(1+ \sum\limits_{P\in S} \mu_P(U_P) \right)^n < \infty.
		\end{align*}
 
	\end{proof}

\begin{rem}
    We briefly compare this with the results in \cite{MSW, MSTW}. There an alternative definition of expected value (respectively higher moment) was used. Namely, under the same assumptions for $(U_P)_{P\in S}$ as in Definition \ref{definition moment} they define
    \begin{equation} \label{def:I}
        I = \{a \in \mO_S^d \mid a \in U_P \text{ for infinitely many } P \in S\}
    \end{equation}
    and $\mathcal{L}(D)^d_I \coloneqq \mathcal{L}(D)^d \setminus I$. 
    Then they define \emph{the $n$-th moment of the system $(U_P)_{P \in S}$} to be 
	\begin{equation} \label{def:mun}
		\lim\limits_{D \in \mathcal{D}_S} \displaystyle{\frac{\sum\limits_{a \in \mathcal{L}(D)^d_I} \lvert \{ P \in S \mid a \in U_{P} \}\rvert^n} {q^{\ell(D)d}}},
	\end{equation} if it exists. Let us call this the renormalized $n$th moment.

This means that they consider, in our language, the moments of the sets $(U_P \setminus I)_{P\in S}$. The restriction to the complement of $I$ prevents the moment from being infinity in a trivial fashion. The rationale in \cite{MSW, MSTW} is that the moment of a random variable with respect to a probability measure does not change when altered on a null set. Furthermore, for system $(U_P )_{P \in S}$ satisfying \eqref{fund_cond_loc_to_glob}, one can show that the set of elements which lie in infinitely many sets $U_P$ has density zero (see the lemma below). Hence, the renormalized moment should be seen as a renormalized version of the more natural notion in Definition \ref{definition moment}. If $(U_P))_{P\in S}$ satisfies \eqref{densitycond}, $(U_P \setminus I)_{P\in S}$ satisfies all the conditions of Theorem \ref{thm:densityFF} and $I$ is closed in all $\widehat{\mathcal{O}}_P$ for $P\in S$, then the $n$th moment of $(U_P \setminus I)_{P\in S}$ coincides with the regularized $n$th moment. 

\end{rem}

\begin{lemma} \label{lm:I zero density}
	Let $F$ be a global function field with full field of constant equal to $\mathbb{F}_q$, $d$ a positive integer, $S$ be a proper, nonempty subset of places of $F$, $S_t$ as defined in \eqref{def:St} and $\mO_S$ the holomorphy ring of $S$. 
For any $P\in S$, let $U_P\subseteq \widehat{\mO}_P^d$ be a Borel-measurable and $I$ defined as in \eqref{def:I}.
\begin{enumerate}
    \item For all $P\in S$ holds
        \begin{align*}
        \mu_P(U_P\setminus I)=\mu_P(U_P).
    \end{align*}
    \item If $(U_P)_{P\in S}$ satisfies \eqref{densitycond}, then
	\begin{equation*}
		\rho_S(I)=\rho_S(\{a \in \mO_S^d \mid a \in U_P \text{ for infinitely many } P \in S\})=0.
	\end{equation*}
\end{enumerate}
\end{lemma}
\begin{proof}
    For the first part we note that $F$ is a finite extension of $\mathbb{F}_q(x)$ and therefore countable. Thus, $I \subseteq \mathcal{O}_S\subseteq F$ is countable too. Recall that $\mu_P$ is a Haar measure and hence, $\mu_P(I)=0$.
    
	If $a \in I$, then for any integer $t$, $a \in U_P$ for some $P \in S_t$. That is 
	\[
		I \subseteq \{ a \in \mO_S^d \mid a \in U_P \text{ for some } P \in S_t\}
	\]
	for all positive integer $t$. So we have 
	\[
		\overline{\rho}_S(I) \leq \lim_{t \to \infty}\overline{\rho}_S \{ a \in \mO_S^d \mid a \in U_P \text{ for some } P \in S_t\} =0,
	\] with the last equality being the from \eqref{fund_cond_loc_to_glob}. So we have $\rho_S(I)=0$ as desired.

\end{proof}

	\section{Applications}

    In this section, we will verify the assumptions of Theorem \ref{Thm:highermoments}, i.e. compute all higher moments, for various examples that were considered in the existing literature. 
	
	\subsection{Coprime \texorpdfstring{$n$}{n}-tuples}
	In this subsection we will compute all higher moments of coprime $n$-tuples. Here a coprime $n$-tuple denotes an $n$-tuple such that all entries are coprime over a specified ring. The computation of the natural density of coprime pairs over the integers is classical and goes back to Mertens and Césaro in the $1870$'s.
	  
	The density of coprime $n$-tuples over holomorphy rings have been calculated in \cite{bib:HolMS}. We now compute the higher moments over holomorphy rings.

    \begin{theorem} \label{thm:mtuples}
        Let $F$ be a global function field with full field of constant equal to $\mathbb{F}_q$. Let $n \geq 2$ be a positive integer. Let $\emptyset \neq S \subsetneq \mathbb{P}_F$ and let $\mO_S$ be the holomorphy ring of $S$. Define the system $U_P = (P \widehat{\mO}_P)^n \setminus \{0\}$ for each $P \in S$. Then all moments exist and are given by \eqref{mean}, where 
        \begin{equation} \label{eq:coprimesP}
            \mu_P(U_P) = q^{-n\deg(P)}.
        \end{equation}
    \end{theorem}

    \begin{proof}
        We show that the system satisfies the assumptions of Theorem \ref{Thm:highermoments}. We first check that condition (\ref{densitycond}) is satisfied using Theorem \ref{condition_verified_polynomials_THEOREM}. Consider the polynomials $f(x_1,x_2, \dots, x_n) = x_1$ and $g(x_1,x_2, \dots, x_n)= x_2$. Then for positive integers $t$, define 
        \begin{align*}
            S_t(f,g) &= \{ a \in \mO_S^n \mid f(a) \in P \text{ and } g(a) \in P \text{ for some } P \in S_t\} \\
            &= \{ (a_1, a_2, \dots, a_n) \in \mO_S^n \mid a_1 \in P \text{ and } a_2 \in P \text{ for some } P \in S_t \}.
        \end{align*} Note that $A_t = \{a \in \mO_S^d \mid a \in U_P \text{ for some } P \in S_t\}$ is a subset $A_t \subset S_t(f,g)$. By Theorem \ref{condition_verified_polynomials_THEOREM}, we have 
        \begin{align*}
            \lim_{t \to \infty} \overline{\rho}_S(A_t) \leq \lim_{t \to \infty}\overline{\rho}_S(S_t(f,g)) = 0.
        \end{align*} So $\lim_{t \to \infty}\overline{\rho}_S(A_t) =0$.

        Next we check that condition (\ref{newcond}) is satisfied. Let $\alpha = 1$. Fix $a = (a_1,a_2, \dots, a_n) \in \mathcal{L}(D)^n \setminus \{(0, \dots, 0)\}$. As $a \neq (0, \dots,0)$, we can without loss of generality assume that $a_1\neq 0$. Now by \cite[Theorem 1.4.11]{bib:stichtenoth2009algebraic}, we have 
        \[
            \sum_{P \in S}\deg(P)v_P(a_1) = - \sum_{P \in \mathbb{P}_F\setminus S} \deg(P)v_P(a_1).   
        \] Recall that $x \in \mathcal{L}(D)$ implies that $v_P(x) \geq -v_P(D)$. So we have for $D\in \mathcal{D}_S$ 
        \begin{align*}
            \sum_{P \in S}\deg(P)v_P(a_1) &= - \sum_{P \in \mathbb{P}_F\setminus S} \deg(P)v_P(a_1) \\
            &\leq \sum_{P \in \mathbb{P}_F\setminus S} \deg(P)v_P(D) \\
            &\leq \deg(D).
        \end{align*} For a lower bound, we have for any constant $c'>0$
        \[
          \sum_{P \in S, \deg(P) \geq c'\deg(D)} \deg(P)v_P(a_1) \leq \sum_{P \in S}\deg(P)v_P(a_1) \leq \deg(D).
        \] 
        Now since $P\setminus = \{x \in F \mid v_P(x) \geq 1\}$ and $a_1\neq 0$, we have 
        \[
            c' \deg(D) \cdot \vert\{P \in S \ : \ \deg(P)\geq \deg(D), a \in U_P \cap \mathcal{L}(D)^n\}\vert \leq \deg(D),
        \] hence condition (\ref{newcond}) is satisfied for any choice of $c'>0$.

        For condition (\ref{newcond2}) and (\ref{newcond3}), let $P_1,\dots, P_r \in S$ be pairwise distinct places. Then we get for $D\in \mathcal{D}_S$
        \[
            \mathcal{L}(D) \cap \bigcap_{j=1}^r (P_j\mO_{P_j}) = \mathcal{L}(D_{P_1, \dots, P_r}),
        \] where $D_{P_1, \dots, P_r}$ is defined by 
        \[
            v_{\widetilde{P}}(D_{P_1, \dots, P_r}) = \begin{cases}
                -1 & \text{ if } \widetilde{P} \in  \{P_1, \dots, P_r\}, \\
                v_{\widetilde{P}}(D) & \text{ otherwise}.
            \end{cases}
        \] 
        If $\deg(P_1), \dots, \deg(P_r) \leq c' \deg(D)$, then we obtain
        \begin{align*}
            \deg(D_{P_1, \dots, P_r}) = \deg(D) - \sum_{j=1}^r \deg(P_j)
            \geq \deg(D) - r c' \deg(D)
            = (1-rc') \deg(D).
        \end{align*}
        Hence, if we pick $0<c'<\frac{1}{r}$, then we can use the Riemann-Roch theorem \cite[Theorem 1.4.17 (b.)]{bib:stichtenoth2009algebraic} and obtain that there exists a constant $C>0$ depending only on $F$ and $c'$ such that for all $D\in \mathcal{D}_S$ with $\deg(D)\geq C$ holds
        \begin{align*}
            \ell(D_{P_1, \dots, P_r})&= \deg(D_{P_1, \dots, P_r}) +1-g
            = \deg(D) +1-g - \sum_{j=1}^r \deg(P_j) \\
            &= \ell(D) -\sum_{j=1}^r \deg(P_j),
        \end{align*}
        where $g$ is the genus of $F$. Thus, we obtain for $\deg(D)\geq C$
        \begin{align*}
            \vert \mathcal{L}(D) \cap \bigcap_{j=1}^r P_j \vert
            =\vert \mathcal{L}(D_{P_1, \dots, P_r}) \vert 
            = q^{\ell(D)} \prod_{j=1}^r q^{-\deg(P_j)}.
        \end{align*}
        Hence, there exists a constant $C>0$ such that for all $D \in \mathcal{D}_S$ with $\deg(D) \geq C$ and all pairwise distinct places $P_1, \dots, P_r\in S$ with $\deg(P_1), \dots, \deg(P_r) \leq \frac{1}{2r} \deg(D)$ holds
        \begin{equation} \label{eq:lowprimeApprox}
            \vert \mathcal{L}(D) \cap \bigcap_{j=1}^r P_j \vert = q^{\ell(D)} \prod_{j=1}^r q^{-\deg(P_j)}.
        \end{equation}
        Thus, we get
        \begin{align*}
            \vert \mathcal{L}(D)^n \cap \bigcap_{j=1}^r U_{P_j} \vert
            = \vert \mathcal{L}(D) \cap \bigcap_{j=1}^r P_j \vert^n
            = q^{n \ell(D)} \prod_{j=1}^r q^{-n \deg(P_j)}.
        \end{align*}		
        So for each $P \in S$, we choose $v_P = q^{-n\deg(P)}$.
	This sequence satisfies \eqref{newcond2}. The series $\sum_{P \in S} q^{-n\deg(P)}$ is dominated by the Zeta function $Z(q^{-n})=\sum_{P \in \mathbb{P}_F} q^{-n\deg(P)}$. For $n \geq 2$, the Zeta function converges by \cite[Prop. 5.1.6]{bib:stichtenoth2009algebraic}, hence \eqref{newcond3}.
    Therefore, we can invoke Theorem \ref{Thm:highermoments}.
    Note that $(P \widehat{\mO}_P)^n$ is a subgroup of $\widehat{\mO}_P^n$ and singletons are null sets, so \eqref{eq:coprimesP} holds true as
    \[
    \mu_P((P \widehat{\mathcal{O}}_P)^n\setminus \{0\})=\mu_P((P \widehat{\mO}_P)^n) = \vert \widehat{\mO}_P^n / (P \widehat{\mO}_P)^n\vert^{-1} = q^{-n\deg(P)}.
    \]

    \end{proof}
	
\subsection{Affine Eisenstein polynomials}
	
    In this subsection, we will compute all higher moments of affine Eisenstein polynomials. The affine Eisenstein polynomials over holomorphy rings can be used to study totally ramified extension, see \cite{functionfields} and the references therein for more details. In said paper, also the density of affine Eisenstein polynomials is computed (see \cite[Theorem 3.6]{functionfields}). The density of the shifted/affine Eisenstein polynomials over number fields have been computed in \cite{micheli2016densityeis} and the higher moments over number fields have been considered in \cite{MSTW}.
    For this section we will by abuse of notation identify polynomials of degree $d$ with the corresponding $(d+1)$-tuple of coefficients.

    Let $\emptyset \neq S \subsetneq \mathbb{P}_F$ and $P\in S$. A polynomial $f(x) \in \mathcal{O}_S[x]$ of degree $d$, say $f(x) = \sum_{j=0}^d a_j x^j$, is said to be $P$-Eisenstein if
    \begin{align*}
        a_d \notin P, a_0 \notin P^2 \text{ and } a_i \in P \ \forall i \in \{0, \dots, d-1\}.
    \end{align*}
    In addition, $f(x)$ is said to be Eisenstein if there exists $P\in S$ such that $f(x)$ is $P$-Eisenstein. We define for $P\in S$
    \begin{equation} \label{eq:UPEisenstein}
        U_P = \left( P \widehat{\mathcal{O}}_P \setminus P^2 \widehat{\mathcal{O}}_P  \right) \times \left(P \widehat{\mathcal{O}}_P \right)^{d-1} \times \left( \widehat{\mathcal{O}}_P \setminus P \widehat{\mathcal{O}}_P \right).
    \end{equation}
    This will be the system for Eisenstein polynomials in $\mathcal{O}_S[x]$ as $U_P \cap \mathcal{O}_S^{d+1}$ represents exactly the $P$-Eisenstein polynomials.

    Next we introduce shifted Eisenstein polynomials. For $P\in S$ we say $f(x) \in \mathcal{O}_S[x]$ is a shifted $P$-Eisenstein polynomial if there exists $t\in \mathcal{O}_S$ such that $f(x+t)$ is a $P$-Eisenstein polynomial. 

    For $t\in \widehat{\mathcal{O}}_P$ we denote by $\sigma_t$ the map 

    \begin{equation} \label{eq:defsigma}
        \sigma_t : \widehat{\mathcal{O}}_P^{d+1} \rightarrow \widehat{\mathcal{O}}_P^{d+1}, f(x) \mapsto f(x+t).
    \end{equation}
    We define for $P\in S$
    \begin{equation} \label{eq:UPshiftedEisenstein}
        V_P = \bigcup_{t\in \mathcal{O}_S} \sigma_t(U_P).
    \end{equation}
    This yields a system for shifted Eisenstein polynomials as $V_P \cap \mathcal{O}_S^{d+1}$ represents the shifted $P$-Eisenstein polynomials.

    Finally we can define affine Eisenstein polynomials. For a commutative ring $R$ we define for 
    $A=\begin{psmallmatrix} \alpha & \beta \\ \gamma & \delta \end{psmallmatrix} \in \text{GL}_{2\times 2}(R) $ 
    and $f(x) \in R[x]$
    \begin{align*}
        (f*A)(x) = (\gamma x + \delta)^d f\left( \frac{\alpha x+ \beta}{\gamma x + \delta}\right).
    \end{align*}
    Let $P$ be a prime ideal in $R$. We call $f(x)\in R[x]$ affine $P$-Eisenstein in $R[x]$ if there exists $A\in \text{GL}_{2\times 2}(R)$ such that $(f*A)(x) \in R[x]$ is $P$-Eisenstein. Furthermore, a polynomial $f(x) \in \mathcal{O}_S[x]$ is called affine Eisenstein if there exists $P\in S$ such that $(f*A)(x)$ is affine $P$-Eisenstein.

    It turns out that only particular affine transformations are needed to realize all affine Eisenstein polynomials. The following lemma is an consequence of \cite[Cor. 3.4]{functionfields}

    \begin{lemma} \label{lm:affine}
    Let $F$ be a global function field, $\emptyset \neq S \subsetneq \mathbb{P}_F$ and $P\in S$.
    
    \begin{enumerate}
        \item Let $\sigma_t$ denote the shift introduced in \eqref{eq:defsigma} and $U_P$ as in \eqref{eq:UPEisenstein}. Let $s,t \in \mathcal{O}_P$, then the following are equivalent:
        \begin{enumerate}
            \item $\sigma_s(U_P) \cap \sigma_t(U_P) \neq \emptyset$
            \item $\sigma_t(U_P) = \sigma_s(U_P)$
            \item $s-t\in P$.
        \end{enumerate}

        \item Let $f\in \mathcal{O}_S[x]$ a polynomial of degree $d \geq 2$.
        Then the following are equivalent:
    \begin{enumerate}
        \item 
        $f(x)$ is affine $P$-Eisenstein
        \item
        There exists $t\in \mathcal{O}_S$ such that $f(x+t)$ is $P$-Eisenstein or $x^d f(1/x)$ is $P$-Eisenstein.
    \end{enumerate}

    \item Let $f\in \mathcal{O}_S[x]$ a polynomial of degree $d \geq 2$, such that $x^d f(1/x)$ is $P$-Eisenstein, then $f(x)$ is not a shifted Eisenstein polynomial.
    \end{enumerate}
    
    \end{lemma} \label{lm:affineEisenstein}
    \begin{proof}
        \begin{enumerate}
            \item As $\sigma_t^{-1} \sigma_s =\sigma_{s-t}$, we can without loss of generality assume that $s=0$. Clearly we have $\sigma_t(U_P)=U_P$ for $t \in P$. Hence, $(c) \Rightarrow (b) \Rightarrow (a)$ hold true.
			
			Let us now assume that $f(x) \in \sigma_t(U_P) \cap U_P$. Write $f(x) = \sum_{j=0}^d a_j x^j$. As $f(x) \in U_P$, we get $a_{d-1} \equiv 0 \text{ (mod } P \widehat{\mathcal{O}}_P)$ and $a_d$ is invertible $\text{ (mod } P \widehat{\mathcal{O}}_P)$. However, $f(x) \in \sigma_t(U_P)$ and thus, looking at the constant coefficient of $f(x+t)$, we get that $$a_d t^d \equiv 0 \text{ (mod } P \widehat{\mathcal{O}}_P)$$
			as $a_0, \dots, a_{d-1}\in P$. Hence, we get $t\in (P \widehat{\mathcal{O}}_P) \cap \mathcal{O}_P = P \mathcal{O}_P$.

            \item Clearly we have $(b) \Rightarrow (a)$. For the other direction we recall that by \cite[Cor. 3.4]{functionfields} that for every affine $P$-Eisenstein polynomial $f(x)$ either $x^d f(1/x)$ is $P$-Eisenstein, or there exists $t\in \mathcal{O}_P$ such that $f(x+t)$ is $P$-Eisenstein. We are left to prove that we can choose $t\in \mathcal{O}_S$. Let $G_P$ be a set of representatives in $\mathcal{O}_P$ of $\mathcal{O}_P/P$ and let $H_P$ be a set of representatives in $\mathcal{O}_S$ of $\mathcal{O}_S/(P \cap \mathcal{O}_S)$. By $1. (c)$ we can write
            \begin{align*}
            \bigsqcup_{t\in H_P} \sigma_t(U_P) = \bigcup_{t\in \mathcal{O}_S} \sigma_t(U_P) \subseteq \bigcup_{\widetilde{t}\in \mathcal{O}_P} \sigma_{\widetilde{t}}(U_P) = \bigsqcup_{\widetilde{t}\in G_P} \sigma_{\widetilde{t}}(U_P).
            \end{align*}
            Now using $1.$ and the fact that $\vert H_P \vert = \vert \mathcal{O}_S/(P\cap \mathcal{O}_S) \vert = \vert \mathcal{O}_P /P \vert = \vert G_P \vert$ we get
            \begin{equation} \label{eq:partitionShifts}
                \bigsqcup_{t\in H_P} \sigma_t(U_P) = \bigcup_{t\in \mathcal{O}_S} \sigma_t(U_P) = \bigcup_{\widetilde{t}\in \mathcal{O}_P} \sigma_{\widetilde{t}}(U_P) = \bigsqcup_{\widetilde{t}\in G_P} \sigma_{\widetilde{t}}(U_P),
            \end{equation}
            which yields the claim.
            \item If $f(x)=\sum_{j=0}^d a_j x^j$ is a shifted $P$-Eisenstein polynomial, then $a_d\notin P$. However, if $x^d f(1/x) = \sum_{j=0}^d a_{d-j} x^j$ is $P$-Eisenstein, then $a_d \in P$.
            
        \end{enumerate}
    \end{proof}

    We define for every $P\in S$

    \begin{align*}
        \text{inv}: \widehat{\mathcal{O}}_P^{d+1} \rightarrow \widehat{\mathcal{O}}_P^{d+1}, f(x) \mapsto x^d f(1/x).
    \end{align*}
    For $P\in S$ let $V_P$ be a in \eqref{eq:UPshiftedEisenstein}, then we define
    \begin{equation} \label{eq:UPaffineEisenstein}
        W_P = V_P \cup \text{inv}(U_P).
    \end{equation}  

    This yields a system for affine Eisenstein polynomials as $W_P \cap \mathcal{O}_S^{d+1}$ represents affine $P$-Eisenstein polynomials by Lemma \ref{lm:affine}.
    
    Now we are ready to compute the higher moments of the Eisenstein polynomials, shifted Eisenstein polynomials and affine Eisenstein polynomials.
    \begin{theorem}
        Let $F$ be a global function field with full field of constants given by $\mathbb{F}_q$. Let $d \geq 3$ be a positive integer. Let $\emptyset \neq S \subsetneq \mathbb{P}_F$ and let $\mO_S$ be the holomorphy ring of $S$. Define the system $(U_P)_{P\in S}, (V_P)_{P\in S}$ and $(W_P)_{P\in S}$ as in \eqref{eq:UPEisenstein}, \eqref{eq:UPshiftedEisenstein}, respectively \eqref{eq:UPaffineEisenstein}. Then all moments exist for all three of these systems and are given by \eqref{mean}, where 
        \begin{equation}
            \mu_P(U_P) = \frac{\left(1-q^{-\deg(P)}\right)^2}{q^{d \deg(P)}}
        \end{equation}
        for the system $(U_P)_{P\in S}$,
        \begin{equation}
            \mu_P(V_P) = \frac{\left(1-q^{-\deg(P)}\right)^2}{q^{(d-1) \deg(P)}}
        \end{equation}
        for the system $(V_P)_{P\in S}$ and
        \begin{equation} \label{eq:sPaffine}
            \mu_P(W_P)= \frac{(1-q^{-\deg(P)})^2\left(1 + q^{\deg(P)}\right)}{q^{d\deg(P)}}
        \end{equation}
        for the system $(W_P)_{P\in S}$.
    \end{theorem}
    \begin{proof}
    First we note that $U_P$ is clopen and that $\inv$ and $\sigma_t$ are homeomorphisms, thus $V_P$ and $W_P$ are clopen too. Hence, the boundary of all those sets are empty. Next, we compute $\mu_P(U_P), \mu_P(V_P)$ and $\mu_P(W_P)$. Clearly we have
    \begin{align*}
        \mu_P(U_P) &= \left(1-\mu_P(P\widehat{\mathcal{O}}_P)\right) \mu_P(P\widehat{\mathcal{O}}_P)^{d-1} \left(\mu_P(P\widehat{\mathcal{O}}_P)-\mu_P(P \widehat{\mathcal{O}}_P)^2 \right) \\
        &= \left(1-q^{-\deg(P)} \right) q^{-(d-1) \deg(P)} \left( q^{-\deg(P)} - q^{-2\deg(P)} \right) \\
        &= \frac{(1-q^{-\deg(P)})^2}{q^{d \deg(P)}}.
    \end{align*}
    Using \eqref{eq:partitionShifts} and the fact that shifts preserve the measure, we obtain
    \begin{align*}
        \mu_P(V_P) = \vert \mathcal{O}_S/(P\cap \mathcal{O}_S) \vert \mu_P(U_P)
        = \frac{\left(1-q^{-\deg(P)}\right)^2}{q^{(d-1) \deg(P)}}.
    \end{align*}
    Using again \eqref{eq:partitionShifts} and Lemma \ref{lm:affineEisenstein} $3.$ we get
    \begin{align*}
        W_P = \text{inv}(U_P) \sqcup \bigsqcup_{t\in H_P} \sigma_t(U_P).
    \end{align*}
    As the shifts and $\inv$ are preserving the measure, we get
    \begin{align*}
        \mu_P(W_P) = (1 + \vert G_P \vert) \mu_P(U_P)
        = (1+ q^{\deg(P)}) \frac{(1-q^{-\deg(P)})^2}{q^{d \deg(P)}}.
    \end{align*}
    We are only showing the claim for the system $(W_P)_{P\in S}$ corresponding to affine Eisenstein polynomials. The other two cases follow as $U_P \subseteq V_P \subseteq W_P$ and hence all the estimates still hold true.

 Let $P \in S$, $D\in \mathcal{D}_S$, and suppose $f \in W_P \cap \mathcal{L}(D)^{d+1}$. 
 
Either $x^d f(1/x) \in U_P \cap \mathcal{L}(D)^{d+1}$ or $f(x) \in V_P\cap \mathcal{L}(D)^{d+1}$. Hence, for any $c'>0$ we have
	\begin{align*}
	    \lvert \{ P \in S & \mid \deg(P) > c'\deg(D), f \in W_P \cap \mathcal{L}(D)^{d+1}\} \rvert \\
	    \leq &\lvert \left\{ P \in S \mid \deg(P) > c'\deg(D), f \in V_P\cap \mathcal{L}(D)^{d+1}\right\} \rvert \\
	    &+ \lvert \{ P \in S \mid \deg(P) > c'\deg(D), x^d f(1/x) \in U_P\cap \mathcal{L}(D)^{d+1}\} \rvert \\
    =:& \  I + II.
	\end{align*} 
    Note that the coefficients of $x^d f(1/x)$ are just a permutation of the coefficients of $f(x)$ and hence, $II \leq I$. Thus, it is enough to bound $I$.

    Let $\disc (f(x))$ denote the discriminant of $f(x)$. We now try to estimate $I$. For this we first consider the case $\disc (f(x)) \neq 0$. Let $b$ be such that $f(x+b) \in U_P$. Since the discriminant is invariant under a shift, $\disc (f(x)) = \disc (f(x+b))$. And so $\disc (f(x)) \in P$. 
    Furthermore, if $f \in \mathcal{L}(D)$, then we have $\disc (f(x)) \in \mathcal{L}(d(d-1)D)$. So we obtain  
	\begin{align*}
	    \vert \{ P \in S \mid  f \in &V_P\cap \mathcal{L}(D)^{d+1}\} \vert \leq  \vert \{ P \in S \mid  \disc (f(x)) \in P\cap \mathcal{L}(d(d-1)D)^{d+1}\} \vert.
	\end{align*} This can be bounded by the same reasoning as for the coprime pairs.
	
	Now suppose $\disc (f(x)) = 0$. Then $f(x)$ is inseparable. so we can write $f(x) = g(x^{p^k})$ for some $k \in \mathbb{N}$, where $g(x)$ is separable. 
    Hence $\disc  (g(x)) \neq 0$. By assumption there exists some $b$ such that $f(x+b)$ is $P$-Eisenstein and thus, we get
	\begin{align*}
	    f(x+b) &= g((x+b)^{p^k})
	    = g(x^{p^k} + b^{p^k})
	\end{align*} is $P$-Eisenstein. Hence, $g(x+b^{p^k})$ is $P$-Eisenstein. This gives us 
        \begin{align*}
	    \vert \{ P \in S \mid  f \in &V_P\cap \mathcal{L}(D)^d\} \vert \leq
	    \vert \{ P \in S \mid  g \in V_P\cap \mathcal{L}(D)^{d/p^k}\} \vert .
	\end{align*}
	Again, using a similar argument as for coprime pairs yields the desired bound. 

        Next we are going to verify conditions \eqref{newcond2}, \eqref{newcond3}. For this we fix some moment $0\neq n\in \mathbb{N}$.
 
    As for coprime pairs we choose $\alpha = 1$ and $c'=\frac{1}{2n}$ and estimate the size of intersections of $\te_{P_1}, \te_{P_2}, \dots, \te_{P_n}$ for 
    distinct $P_1, \dots, P_n$ with $\deg(P_j) \leq \frac{1}{2n}\deg(D)$ for $j \in \{ 1, 2, \ldots, n\}$ for $D\in \mathcal{D}_S$ with $\deg(D)\geq C$ for some constant $C$ depending only on $n,d$.
    Let $f \in \bigcap_{i=1}^n \te_{P_i}$ 
    for each $P_i$, either $f(x+t_i)$ is $P_i$-Eisenstein for some $t_i\in \mathcal{O}_S$, or $x^df(1/x)$ is $P_i$-Eisenstein. If $f(x + t_i)$ is $P_i$-Eisenstein then we have
    \begin{align*}
     a_d(x+t_i)^d + a_{d-1}(x+t_i)^{d-1}+\cdots+a_0 = a_d'x^d+a_{d-1}'x^{d-1}+\cdots+a_0'
    \end{align*} for some $a_d', a_{d-1}', \ldots, a_0' \in \mathcal{O}_S$. 
    Thus, we can express $a_{d-1}', a_{d-2}', a_{d-3}'$ as functions of $a_{d-3}, \dots, a_d$ and $t_i$
    \begin{align*}
        a_{d-1}' &= a_{d-1} + \binom{d}{1}t_i a_d\\
        a_{d-2}' &= a_{d-2} + \binom{d-1}{1}t_i a_{d-1} +    \binom{d}{2}t_i^2 a_d\\
        a_{d-3}' &= a_{d-3} + \binom{d-2}{1}t_ia_{d-2} + \binom{d-1}{2}t_i^2 a_{d-1} + \binom{d}{3}t_i^3 a_d.
    \end{align*}

    As $f(x+t_i)$ is $P_i$-Eisenstein, we get:
    \begin{align*}
       a_{d-1} + \binom{d}{1}t_i a_d &\equiv 0 &&\Mod {P_i}\\
        a_{d-2} + \binom{d-1}{1}t_i a_{d-1} +    \binom{d}{2}t_i^2 a_d &\equiv 0 &&\Mod {P_i}\\
        a_{d-3} + \binom{d-2}{1}t_ia_{d-2} + \binom{d-1}{2}t_i^2 a_{d-1} + \binom{d}{3}t_i^3 a_d &\equiv 0 &&\Mod {P_i}.
    \end{align*}
    Therefore, given $a_d$ and $t_i$, from the first equation, $a_{d-1} \Mod {P_i}$ is fixed, denote it $\overline{a}_{d-1}(a_d, t_i)$. Given the second equation $a_{d-2} \Mod {P_i}$ is also fixed. This is also true for $a_{d-3}$, denote them by $\overline{a}_{d-2}(a_d, t_i)$ 
    and $\overline{a}_{d-3}(a_d, t_i)$ respectively. 
    If $x^df(1/x)$ is $P_i$-Eisenstein then we must have $a_{d-1}, a_{d-1}, a_{d-3} \in P_i$.
    
    We partition $P_1, P_2, \ldots, P_n$ into $Q_1, Q_2, \ldots, Q_{n-k}$ and $S_1, S_2, \dots, S_k$, where $k\in \{0, n\}$ are valid choices. Suppose $x^df(1/x)$ is Eisenstein with respect to $Q_1, Q_2, \ldots, Q_j$ and $f(x)$ is shifted Eisenstein with respect to $S_1, S_2, \dots, S_k$. We count the number of such $f$. The restriction on $x^nf(1/x)$ implies that $a_{d-1}$ and $ a_{d-2}$ both are in $Q_1, Q_2, \ldots, Q_{n-k}$. Then, since $f$ is shifted Eisenstein with respect to each $S_i$, for each $S_i$ there exists some $t_{i}$ such that $f(x+t_{i})$ is $S_i$-Eisenstein. For each $i$, the number of such $t_{i}$ that can be chosen is $|\mO_{S_i}/S_i| = q^{\deg(S_i)}$, since every polynomial Eisenstein with respect to $S_i$ shifted by an element of $S_i$ is again Eisenstein with respect to $S_i$ (see Lemma \ref{lm:affineEisenstein}). 
    This implies that the coefficients of $f$ satisfy the following system of equations 
    \begin{align} \label{eq:lowprimeAffine}
        a_{d-1}&= \overline{a}_{d-1}(a_d, t_{i}) &&\Mod {S_i} \nonumber\\
        a_{d-2}&= \overline{a}_{d-2}(a_d, t_{i}) &&\Mod {S_i} \nonumber
        \\
        a_{d-3}&= \overline{a}_{d-3}(a_d, t_{i}) &&\Mod {S_i}\\
        a_{d-1}&= 0 &&\Mod {Q_m} \nonumber\\
        a_{d-2}&= 0 &&\Mod {Q_m} \nonumber
    \end{align}
    for each $S_i$ for all $i \in \{ 1, \ldots, k\}$ and each $Q_m$ with $m \in \{ 1, \ldots, n-k \}$. Using the the Chinese Remainder Theorem and the fact that $P_i \cap \mathcal{O}_S$ and $P_j\cap \mathcal{O}_S$ are coprime for $P_i \neq P_j$ (by \cite[Prop. 3.2.9.]{bib:stichtenoth2009algebraic}), we see that the coefficients $a_{d-1}, a_{d-2}$ are uniquely determined in $\mathcal{O}_S/(\mathcal{O}_S \cap \bigcap_{i=1}^k S_i \cap \bigcap_{m=1}^{n-k} Q_m)$ and $a_{d-3}$ is uniquely determined in $\mathcal{O}_S/(\mathcal{O}_S \cap \bigcap_{i=1}^k S_i)$, once we have fixed $a_d\in \mathcal{L}(D)$ and $t_i\in \mathcal{O}_S/P_i$. This readily implies that
    \begin{align*}
        &\vert \{ a\in \mathcal{L}(D)^{d+1} \ : \ f(x) \in \bigcap_{i=1}^k V_{P_i}, x^d f(1/x) \in \bigcap_{m=1}^{n-k} U_{Q_m} \} \vert \\
        &\leq \vert \mathcal{L}(D)\vert \cdot q^{\sum_{i=1}^k \deg(S_i)} \cdot \vert \mathcal{L}(D) \cap \bigcap_{i=1}^k S_i \cap \bigcap_{m=1}^{n-k} Q_m \vert^2 \cdot \vert \mathcal{L}(D) \cap \bigcap_{i=1}^k S_i \vert \cdot \vert \mathcal{L}(D)\vert^{d-3} \\
        &=\vert \mathcal{L}(D)\vert^{d+1} \prod_{j=1}^n q^{-2\deg(P_j)} ,
    \end{align*}    
 
    where we used for the for the first inequality that $a_d\in \mathcal{L}(D), a_0, \dots, a_{d-4} \in \mathcal{L}(D) $ and that there are $q^{\deg(S_i)}$ choices for $t_i$.  With \eqref{eq:lowprimeApprox} we can pass to the third line. Summing over all possible partitions, we get
    \begin{align*}
        \vert \mathcal{L}(D)^{d+1} \cap \bigcap_{j=1}^n W_{P_j} \vert 
        \leq 2^n q^{\ell(D) (d+1)} \prod_{j=1}^n q^{-2 \deg(P_j)}.
    \end{align*}

    Since $\sum_{P \in \mathbb{P}_F} q^{-2\deg(P)}$ is dominated by the Zeta function $Z(q^{-2})$, and since $Z(q^{-2})$ converges (see \cite[Prop. 5.1.6.]{bib:stichtenoth2009algebraic}), $\sum_{P \in \mathbb{P}_F} q^{-2\deg(P)}$ converges too.

    \end{proof}

	\subsection{Rectangular unimodular matrices}
 
    In this subsection we compute all higher moments of rectangular unimodular matrices over function fields. Rectangular unimodular matrices have already been considered in the literature in similar situations. Namely, their density over number fields has been calculated in \cite{densitiesunimodular}; their density over function fields has been done in \cite{functionfields}, and their expected value over rationals has been established in \cite{MSW}.

    Let us recall the definition of rectangular unimodular matrices over a Dedekind domain.
    Let $\mathcal{D}$ be a Dedekind domain and $n<m\in \mathbb{N}$. A matrix $M \in \Mat_{n\times m}(\mathcal{D})$ is called rectangular unimodular, if and only if $M \mod P$ has full rank for all non-zero prime ideal $P$ of $\mathcal{D}$ \cite[Proposition 3]{densitiesunimodular}. This is equivalent to saying that the matrix has a basic minor which is not contained in $P\widehat{\mathcal{O}}_P$, where a basic minor of a matrix is the determinant of a square submatrix that is of maximal size.
    Note that rectangular unimodular matrices in the case $n=1$ correspond to coprime pairs.
	We will use \eqref{Thm:highermoments} and ideas of \cite{MSW} to compute all higher moments of rectangular unimodular matrices. 
	
	\begin{theorem}
    Let $n,m$ be positive integers such that $n < m$ and $F$ be a global function field with full field of constants equal to $\mathbb{F}_q$ and $\emptyset \neq S\subsetneq \mathbb{P}_F$. For any $P \in S$, let $V_P$ be the set of matrices in $\text{Mat}_{n\times m}(\widehat{\mathcal{O}}_P)$ for which the ideal generated by the basic minors is contained in $P\hat{\mathcal{O}}_P$ and let $I$ be the set of matrices contained in infinitely many $V_P$. We define $U_P=V_P \setminus I$.

    Then the system $(U_P)_{P\in S}$ satisfies the conditions of Theorem \ref{Thm:highermoments}, and the higher moments are given by \eqref{mean}, where 
     \begin{equation}\label{s_eta}
    	    \mu_P(U_P)=1-\prod\limits_{i=0}^{n-1}\left(1-q^{-deg(P)(m-i)}\right).
     \end{equation}

    \end{theorem}
    \begin{proof}
     We have \eqref{s_eta} due to the computation in the proof of \cite[Theorem 4.4]{functionfields} and the fact that $\mu_P(U_P \cap I)=0$ (by Lemma \ref{lm:I zero density}). We are left to check that the assumptions of Theorem \eqref{Thm:highermoments} are satisfied.
     
     We start by noting that $V_P$ is clopen and $I$ is closed, hence $\mu_P(\partial U_P)\leq \mu_P(I)=0$ by Lemma \ref{lm:I zero density}. Condition \eqref{densitycond} is satisfied as shown in \cite[Theorem 13]{functionfields}. We want to show Condition \eqref{newcond} is satisfied for $\alpha=1$. 
     Let $M\in \Mat_{n\times m}(\mathcal{O}_S)\setminus I$. Fix a submatrix of $M$ corresponding to some basic minor of $M$, and denote it by $A$. When $M\in U_P$, by definition, there exist some $A$ such that $\det(A)\neq 0$. We choose such an $A$. One can see that when $M\in U_P \cap \LD^{n\times m}$, then $A\in \LD^{n\times n}$. Thus, 
     \begin{align*}
         v_P(\det(A))
         &=v_P\left(\sum_{\sigma\in S_n} \text{sgn} (\sigma)\prod^{n}_{i=1}A_{i,\sigma(i)}\right)\\
         &\geq \min_{\substack{\sigma\in S_n \\ \prod^{n}_{i=1}A_{i,\sigma(i)}\neq 0}} v_P\left(\prod^{n}_{i=1}A_{i,\sigma(i)}\right)\\
         &\geq -n v_P(D).
     \end{align*}
     This implies, as for coprime pairs, that for $D\in \mathcal{D}_S$ we have
    \begin{align*}
        \sum_{P \in S}\deg(P)v_P(\det(A)) 
        &= - \sum_{P \in \mathbb{P}_F\setminus S} \deg(P)v_P(\det(A))\\
        &\leq \sum_{P \in \mathbb{P}_F\setminus S} \deg(P) n  v_P(D)\\
        &\leq n\deg(D).
    \end{align*}
    Hence, we can obtain a lower bound, for any constant $c'>0$,
    \begin{equation*}
        \sum_{\substack{P \in S\\ \deg(P)\geq c'\deg(D)}} \deg(P)v_P(\det(A)) \leq \sum_{P \in S}\deg(P)v_P(\det(A)) \leq n\deg(D).
    \end{equation*}
    Since $0\neq \det(A)\in P\widehat{\mathcal{O}}_P$, 
    and $P= \{x \in F \mid v_P(x) \geq 1\}$, we have 
    \begin{equation*}
        c'\deg(D)\vert\{P \in S: \deg(P)> \deg(D), M \in U_P \cap \LD^{n\times m}_{I}\}\vert \leq n\deg(D).
    \end{equation*}
    Hence condition (\ref{newcond}) is satisfied for any $c'>0$. 
    
    Now we check condition (\ref{newcond2}) and (\ref{newcond3}). For each $A\in{\LD}^{n\times n}$ such that $0\neq \det(A)\in\bigcap_{i=1}^{r}P_{i}$, by definition, there exists $M\in \bigcap_{i=1}^{r} U_{P_{i}}\cap {\LD_{I}}^{n\times m}$ containing $A$ as a submatrix. There are less or equal than $\binom{m}{n}\vert \mathcal{L}(D) \vert^{nm-n^2}$ such choices per matrix $M$. So we have
    \begin{equation*}
    	\lvert \bigcap_{i=1}^{r} U_{P_{i}}\cap {\LD_{I}}^{n\times m}\rvert
        \leq \binom{m}{n}\vert \mathcal{L}(D) \vert^{nm-n^2}\lvert \{A\in {\LD}^{n\times n} \mid 0\neq\det(A)\in \bigcap_{i=1}^{r}P_{i}\}\rvert. 
    \end{equation*}
    Fix an arbitrary $D\in\mathcal{D}_S$. Define $\phi_1$ to be the inclusion map $\Mat_{n \times n}(\LD) \rightarrow \Mat_{n \times n}(\OS)$, and $\phi_2$ to be the natural quotient map $\Mat_{n \times n}(\OS) \rightarrow \Mat_{n \times n}(\sfrac{\OS}{\bigcap_{i=1}^{r}P_{i}})$. Let $\phi=\phi_2 \circ \phi_1$. As $\{A\in {\LD}^{n\times n} \mid 0\neq\det(A)\in \bigcap_{i=1}^{r}P_{i}\}$ is a subset of $\Mat_{n \times n}(\LD)$, we get
    \begin{align*}
        &\lvert \{A\in {\LD}^{n\times n} \mid 0\neq\det(A)\in \bigcap_{i=1}^{r}P_{i}\}\rvert\\
        &\leq \lvert \{B\in \Mat_{n \times n}(\sfrac{\OS}{\bigcap_{i=1}^{r}P_{i}}) \mid \det(B)=0\}\rvert \cdot \lvert \ker(\phi) \rvert\\
        &=\lvert \{B\in \Mat_{n \times n}(\sfrac{\OS}{\bigcap_{i=1}^{r}P_{i}}) \mid \det(B)=0\}\rvert \cdot \lvert \LD \cap \bigcap_{i=1}^{r}P_{i} \rvert^{n^2}.
    \end{align*}
    Recall that $P_i \cap \mathcal{O}_S$ and $P_j\cap \mathcal{O}_S$ are distinct maximal ideals in $\mathcal{O}_S$ for $P_i \neq P_j$ (see \cite[Prop. 3.2.9.]{bib:stichtenoth2009algebraic}) and therefore, by the Chinese Remainder Theorem, the following is an isomorphism of $\OS$-modules 
    \begin{align*}
    	\pi: \Mat_{n\times n}(\sfrac{\OS}{\bigcap_{i=1}^{r}P_{i}}) &\rightarrow \prod\limits_{i=1}^{r}\Mat_{n\times n}(\sfrac{\OS}{P_{i}}), \\ \left(a_{jk} + \bigcap_{i=1}^r P_i \right)_{1\leq j,k \leq n} &\mapsto \left( (a_{jk}+P_1)_{1\leq j,k\leq n}, \dots, (a_{jk} + P_r)_{1\leq j,k \leq n} \right).
    \end{align*}
    Clearly we have that $\det(a_{jk}+\bigcap_{i=1}^{r}P_{i})=0 \mod \bigcap_{i=1}^{r}P_{i}$ if and only if $\det(a_{jk}+P_{i})=0 \mod P_{i}$ for all $P_i$. Therefore, 
    \begin{equation*}
        \lvert \{B\in \Mat_{n \times n}(\sfrac{\OS}{\bigcap_{i=1}^{r}P_{i}}) \mid \det(B)=0\}\rvert
        =\prod_{i=1}^{r} \lvert \{B_i\in \Mat_{n \times n}(\sfrac{\OS}{P_{i}}) \mid \det(B_i)=0\}\rvert.
    \end{equation*}
    We know that the quotient ring $\sfrac{\OS}{P_i}$ is isomorphic to $\mathbb{F}_{q^{\deg(P_i)}}$ (see \cite[Prop. 3.2.9.]{bib:stichtenoth2009algebraic}). So we have 
    \begin{align*}
    &\lvert\{B_i\in \Mat_{n \times n}(\sfrac{\OS}{P_{i}}) \mid \det(B_i)=0\}\rvert\\
    &=\lvert\Mat_{n \times n}(\mathbb{F}_{q^{\deg(P_i)}})\setminus \text{GL}_{n \times n}(\mathbb{F}_{q^{\deg(P_i)}})\rvert\\
    &=q^{deg(P_i)n^{2}}-\prod_{k=0}^{n-1}(q^{\deg(P_i)n}-q^{\deg(P_i)k})\\
    &\leq 2^n q^{\deg(P_i)n(n-1)}.
    \end{align*}
    By \eqref{eq:lowprimeApprox}, there exist a constant $C>0$ such that for all $D\in \mathcal{D}_S$ with $\deg(D)\geq C$, it holds that
    \begin{equation*}
       \lvert \LD \cap \bigcap_{i=1}^{r}P_{i} \rvert
       = q^{\ell(D)} \prod_{j=1}^r q^{-\deg(P_j)}.
    \end{equation*}
    Now, combine everything, we have 
    \begin{align*}
    &\lvert \bigcap_{i=1}^{r} U_{P_{i}}\cap \mathcal{L}(D)^{n\times m}\rvert\\
    &\leq \binom{m}{n} q^{(nm-n^2)\deg(D)}
    q^{\ell(D)n^2} \prod_{j=1}^r q^{-n^2\deg(P_j)}
    \prod_{i=1}^{r}\big(2^n q^{\deg(P_i)n(n-1)}\big)\\
    &= C' q^{nm \ell(D)} \prod_{j=1}^r q^{-n\deg(P)}
    \end{align*}
    for a constant $C'>0$ depending only on $n,m$. Observe that $\sum\limits_{P\in\mathbb{P}_F}q^{-n\deg(P)}$ is the Zeta function $Z(q^{-n})$. By \cite[Prop. 5.1.6.]{bib:stichtenoth2009algebraic}, it converges when $n>1$. The case $n=1$ corresponds to coprime $m$-tuples and is covered by Theorem \ref{thm:mtuples}.
    \end{proof}
	
	\section*{Acknowledgments}
	
	We would like to thank the referee for their helpful comments. Andy Hsiao and Severin Schraven gratefully acknowledge support by NSERC of Canada.
	
	\bibliographystyle{plain}
\bibliography{biblio}
\end{document}